\documentclass[10pt,twoside,a4paper,american]{article}

\pdfoutput=1 
\usepackage[expansion=true]{microtype}
\usepackage[T1]{fontenc}
\usepackage[utf8]{inputenc}
\usepackage{enumitem}
\usepackage{titlesec}

\titlespacing*{\section}
{0pt}{2.2ex plus 0.5ex minus .2ex}{1ex plus .2ex}
\titlespacing*{\subsection}
{0pt}{1.5ex plus 0.5ex minus .2ex}{0.8ex plus .2ex}
\usepackage{times}
\usepackage[colorlinks=true,hyperindex=true]{hyperref}
\usepackage{fancyhdr}
\usepackage{amsmath,amsfonts,amssymb}
\usepackage{amsthm}

\usepackage{datetime2}
\usepackage[margin=2.95cm]{geometry}
\usepackage{color,xcolor}

\colorlet{darkblue}{blue!50!black}

\hypersetup{
    colorlinks,%
    citecolor=darkblue,%
    filecolor=red,%
    linkcolor=darkblue,%
    urlcolor=magenta,%
    pdfnewwindow=true,%
    pdfstartview={FitH}
}

\newtheorem{theorem}{Theorem}[section]
\newtheorem{proposition}[theorem]{Proposition}
\newtheorem{lemma}[theorem]{Lemma}
\newtheorem{corollary}[theorem]{Corollary}

\theoremstyle{definition}
\newtheorem{definition}[theorem]{Definition}
\newtheorem{remark}[theorem]{Remark}

\def\rr{{\mathbb R}}

\def\nn{{\mathbb N}}

\def\cA{{\cal A}}

\def\cC{{\cal C}}
\def\cF{{\cal F}}

\def\cP{{\cal P}}
\newcommand{\CL}{{\cal L}}
\newcommand{\GG}{{\cal G}}

\def\dd{{\textup d}}
\def\d{{\textup d}}
\def\ie{\textit{i.e., }}

\newcommand\abstractbox[2]{{
\begin{center}
\begin{minipage}{0.9\textwidth}\small
\noindent{\bf #1} #2
\end{minipage}\end{center}}}

\numberwithin{equation}{section}
\begin{document}
\def\today{}
\title{Additive, almost additive and asymptotically\\  additive potential sequences are  equivalent}

\author{ Noé~Cuneo\\~\\
{\normalsize Laboratoire de Probabilit\'es, Statistique et Mod\'elisation (LPSM),} \\[-0.5mm]{\normalsize Universit\'e de Paris -- Sorbonne Universit\'e -- CNRS,}\\[-0.5mm]{\normalsize  75205 Paris CEDEX 13 France}}

\maketitle

\abstractbox{Abstract.}{
Motivated by various applications and examples, the standard notion of potential for dynamical systems has been generalized to almost additive and asymptotically additive potential sequences, and the corresponding thermodynamic formalism, dimension theory and large deviations theory have been extensively studied in the recent years. In this paper, we show that every such potential sequence is actually equivalent to a standard (additive) potential in the sense that there exists a continuous potential with the same topological pressure, equilibrium states, variational principle, weak Gibbs measures, level sets (and irregular set) for the Lyapunov exponent and large deviations properties. In this sense, our result shows that almost and asymptotically additive potential sequences do not extend the scope of the theory compared to standard potentials, and that many results in the literature about such sequences can be recovered as immediate consequences of their counterpart in the additive case. A corollary of our main result is that all quasi-Bernoulli measures are weak Gibbs.


}

\medskip 


\section{Introduction}

Let $(X, d)$ be a compact metric space and let $C(X)$ be the space of real-valued continuous functions on $X$ endowed with the sup-norm $\|\cdot\|_\infty$ (the compactness assumption will be lifted in Section~\ref{sec:gennoncompactdisc}). Let $T: X\to X$ be a continuous map. We denote by $\cP(X)$ the space of Borel probability measures on $X$, and by $\cP_T(X)$ the set of $T$-invariant elements of $\cP(X)$.

In the standard (additive) thermodynamic formalism (see for example \cite{walters1982, KH1995, ruelle2004}), the {\em potential} is a function $f \in C(X)$ and the central role is played by the sequence of functions $(S_n f)_{n\geq 1}$, where
\begin{equation}\label{eq:defSnf}
	S_n f = \sum_{k=0}^{n-1} f\circ T^k.
\end{equation}
The sequence $(S_n f)_{n\geq 1}$ is called {\em additive}, which refers to the property that $S_{n+m}f = S_n f + (S_m f)\circ T^n$, $n,m\geq 1$. The  {\em non-additive} thermodynamic formalism, which was first introduced in \cite{falconer_subadditive88}, consists in replacing the sequence $(S_n f)_{n\geq 1}$ with a more general sequence $(f_n)_{n\geq 1}\subset C(X)$ subject to some conditions. Several classes of sequences $(f_n)_{n\geq 1}$ have been studied in the literature, and in particular a lot of attention has been devoted to {\em almost additive} and {\em asymptotically additive} sequences of potentials:
\begin{definition}\label{def:definitionaa}
	A sequence  $(f_n)_{n\geq 1} \subset C(X)$ is said to be
\begin{enumerate}
	\item[(i)] {\em almost additive}  if there exists $C\geq 0$ such that 
\begin{equation}\label{eq:defalmostadd}
	\|f_{n+m}-f_n - f_m\circ T^n\|_\infty\leq C, \quad n,m\geq 1;
\end{equation}
	\item[(ii)] {\em asymptotically additive}  if
\begin{equation}\label{eq:defasymostadd}
	\inf_{f\in C(X)}\limsup_{n\to\infty} \frac 1 n \|f_n - S_n f\|_\infty = 0.
\end{equation}
\end{enumerate}
\end{definition}
Almost additive sequences were first introduced in \cite{barreira-2006,Mummert06} and asymptotically additive sequences in \cite{feng_huang_2010} (see in particular Proposition A.5 (iv) therein for a formulation of the definition as above). It is easy to realize that additive sequences satisfy \eqref{eq:defalmostadd} and \eqref{eq:defasymostadd}. Moreover, it is well known that every almost additive sequence is asymptotically additive, since \eqref{eq:defalmostadd} implies that $\lim_{k\to \infty} \limsup_{n\to\infty}\frac 1 n \|f_n - S_n (\frac {f_k}k)\|_\infty = 0$, see for example \cite[Proposition A.5 (ii)]{feng_huang_2010} or \cite[Proposition 2.1]{zhao_asymptotically_2011}. More references about almost and asymptotically additive sequences will be given below.

The main result of this paper is that the infimum in \eqref{eq:defasymostadd} is actually reached, which answers the open question in \cite[Remark 6.6]{CJPS_topo}, and shows that almost and asymptotically additive sequences are much closer to being additive than previously thought.
\begin{theorem}\label{th:maintheorem}
Let $(f_n)_{n\geq 1}$ be an asymptotically additive (or almost additive) sequence of potentials. Then, there exists $f\in C(X)$ such that
\begin{equation}\label{eq:asympaddnew}
\lim_{n\to\infty}\frac 1n\|f_n - S_n f \|_\infty = 0.
\end{equation}
\end{theorem}
The proof of Theorem~\ref{th:maintheorem} is elementary and will be provided in Section~\ref{sec:proofmain}. As discussed in Remark~\ref{rem:afterproof} below, the function $f$ in \eqref{eq:asympaddnew} is unique up to some (standard) equivalence relation.
We also note that any sequence $(f_n)_{n\geq 1}\subset C(X)$ satisfying  \eqref{eq:asympaddnew} for some $f\in C(X)$ is obviously  asymptotically additive. Thus, $(f_n)_{n\geq 1}$ is asymptotically additive if and only if it satisfies \eqref{eq:asympaddnew} for some  $f\in C(X)$. 

The relation \eqref{eq:asympaddnew} means that the sequence $(f_n)_{n\geq 1}$ is additive up to sublinear corrections. More precisely, if $ (f_n)_{n\geq 1} \subset C(X)$ is almost or asymptotically additive, then there exists $f\in C(X)$ such that
\begin{equation}\label{eq:fsublinearcorr}
	f_n = S_n f + u_n, \quad n\geq 1,
\end{equation}
where $u_n\in C(X)$ satisfies $\lim_{n\to\infty} n^{-1}\|u_n \|_\infty = 0$.
As discussed in detail in Section~\ref{sec:nonaddpot}, this implies that the sequence $(f_n)_{n\geq 1}$ has the same topological pressure, equilibrium states, variational principle, weak Gibbs measures, level sets (and irregular set) for the Lyapunov exponent and large deviations properties as the potential $f$. This makes much of the theory of almost additive and asymptotically additive sequences redundant in the sense that it becomes an immediate consequence of the standard theory in the additive case and that it does not extend its scope (at least as far as the properties mentioned above are concerned).

However, we stress that the almost additive and asymptotically additive conditions remain very relevant for applications: by Theorem~\ref{th:maintheorem}, they can be seen as sufficient conditions for the additive theory to apply. Moreover, in many cases, the sequence $(f_n)_{n\geq 1}$ is given, but the additive potential $f$ is not known explicitly (Theorem~\ref{th:maintheorem} only proves its existence as a limit in a quotient space). Thus, working with $(f_n)_{n\geq 1}$ instead of $(S_n f)_{n\geq 1}$ may still be required for some explicit computations, viewing $(f_n)_{n\geq 1}$ as an explicitly known approximation of $(S_n f)_{n\geq 1}$.

In 1D statistical mechanics, when $f_n$ is (minus a constant times) the Hamiltonian of a finite-volume system of size $n$ with absolutely summable interactions (see for example \cite[Section 3.2]{ruelle2004}), it is well known that \eqref{eq:asympaddnew} holds when $f$ is the energy per site, and the sublinear corrections $u_n$ in \eqref{eq:fsublinearcorr} are viewed as boundary terms. In a sense, the construction in the present paper yields an $f$ which plays a similar role for general dynamical systems, provided $(f_n)_{n\geq 1}$ is almost or asymptotically additive.

The paper is organized as follows. We prove a slightly more general version of Theorem~\ref{th:maintheorem} in Section~\ref{sec:proofmain}. Section~\ref{sec:nonaddpot} is then mostly intended for the reader less familiar with the non-additive thermodynamic formalism. There, we give the basic definitions of the non-additive thermodynamic formalism and show how they are reduced to the usual additive definitions as a consequence of Theorem~\ref{th:maintheorem}. We also give detailed references about almost and asymptotically additive sequences. In Section~\ref{sec:corollariesetc}, we discuss some less obvious consequences and extensions of Theorem~\ref{th:maintheorem}: In Section~\ref{sec:weakGibbsQB}, we prove that all weakly coupled and quasi-Bernoulli measures are weak Gibbs, which leads to a natural open question about almost additive potential sequences and Gibbs measures (see Section~\ref{sec:openquestionAlmostadd}). In Section~\ref{sec:gennoncompactdisc}, we generalize the main result to non-compact spaces and discontinuous potentials, and we discuss some consequences of the moderate variation condition in Section~\ref{sec:moderatevar}.

\noindent{\bf Acknowledgements}. I would like to thank V.~{Jak\v si\'c} and A.~Shirikyan for providing a clever improvement in the proof of the main result, and C.-E.~Pfister for suggesting a reformulation that led to Theorem~\ref{thm:generalizedversion} below. I am also very thankful to  N.~Dobbs, J.-P.~Eckmann, B.~Fernandez and C.-A.~Pillet for stimulating conversations and useful comments about the manuscript. 

\section{Proof of the main result}\label{sec:proofmain}

We prove here a slight generalization of Theorem~\ref{th:maintheorem}. Throughout this section, we assume that $(V, \|\cdot\|)$ is a normed vector space and that $W\subset V$ is a subspace such that $(W, \|\cdot\|)$ is a Banach space. We let $K: W\to W$ be a contraction (\ie $\|Kf\| \leq \|f\|$ for all $f\in W$), and for $f\in W$, $n\geq 1$, we write $S_nf = \sum_{k=0}^{n-1} K^k f$.
\begin{theorem}\label{thm:generalizedversion}
 Let  $(f_n)_{n\geq 1}\subset V$ be a sequence such that
\begin{equation}\label{eq:defasymostaddgen}
	\inf_{f\in W} \limsup_{n\to\infty}\frac 1n\|f_n - S_n f\| = 0.
\end{equation}
Then, there exists $f\in W$ such that
\begin{equation}\label{eq:asympaddnewgen}
	\lim_{n\to\infty}\frac 1n\|f_n - S_n f\| = 0.
\end{equation}
\end{theorem}

Theorem~\ref{th:maintheorem} corresponds to the special case $V = W = C(X)$, $\|\cdot\| = \|\cdot\|_\infty$ and $Kf = f\circ T$.
The general formulation given in Theorem~\ref{thm:generalizedversion} allows, for example, to consider also functions $f_n$ that are possibly discontinuous, by choosing $V = B(X)$, where $B(X)$ is the space of bounded functions, and $W = C(X)$, with still  $\|\cdot\| = \|\cdot\|_\infty$ and $Kf = f\circ T$. Furthermore, Theorem~\ref{thm:generalizedversion} will turn out to be a convenient formulation when we deal with non-compact spaces $X$ in Section~\ref{sec:gennoncompactdisc}.

The main ingredient of the proof is the following standard construction (see Remark~\ref{rem:setupmainthm} for comments). Let $\CL = {\rm cl}\{h - Kh: h\in W\} \subset W$, where the closure is taken in the Banach space $(W, \|\cdot\|)$. Consider then the equivalence relation $\sim$ on $W$ defined by $f\sim g$ whenever $f -g \in \CL$. Let then  $\widetilde f = \{g\in W: f\sim g\}$ be the equivalence class of $f\in W$. Since $\CL$ is closed in $(W, \|\cdot\|)$, the quotient space $W/{\sim}$ is a Banach space (see for example \cite[Proposition 1.35]{fabian2011banach}) with respect to the norm $\|\cdot\|_*$ defined by
\begin{equation}\label{eq:normstar}
\|\widetilde f\|_{*} = \inf_{g\in \widetilde  f} \|g\|, \quad \widetilde f \in W/{\sim}.
\end{equation}
In order to prove Theorem~\ref{thm:generalizedversion}, we need the following lemma (see for example \cite[Proposition 2.35]{van_enter_regularity_1993}):
\begin{lemma}\label{lem:norme1nsn}Let $f\in W$. Then,
\begin{equation}\label{eq:norm1n}
	 \lim_{n\to\infty} \frac 1 n \|S_n f\| =\|\widetilde f\|_*.
\end{equation}
\end{lemma}
\begin{proof}{}
We start by observing that for each $n\geq 1$,
\begin{equation}\label{eq:1nSnequivf}
\frac 1 n S_n f \sim f.	
\end{equation}
Indeed, one can check that $f - \frac 1 n S_n f = h_n - K h_n$ with $h_n =\frac 1n\sum_{i=1}^{n-1}S_i f$, so that \eqref{eq:1nSnequivf} holds. 

By \eqref{eq:1nSnequivf} and the definition of $\|\cdot\|_{*}$, we have  $ \frac 1n\|S_n f \| \geq \|\widetilde f\|_*$ and thus   $\liminf_{n\to\infty} \frac 1 n \|S_n f\| \geq \|\widetilde f\|_{*}$.
To prove a bound in the opposite direction, we observe that for every $\varepsilon > 0$, there exists $h_\varepsilon \in W$ such that $\|f+h_\varepsilon-K h_\varepsilon\| \leq \|\widetilde f\|_{*}+\varepsilon$. It follows that
\begin{align*}
	\|S_n f \| &\leq 	\|S_n (f+h_\varepsilon-Kh_\varepsilon ) \| + 	\|S_n (h_\varepsilon-Kh_\varepsilon )\| \\
&\leq n\|f+h_\varepsilon-Kh_\varepsilon  \| + \|h_\varepsilon- K^nh_\varepsilon\| \leq n(\|\widetilde f\|_{*}+\varepsilon) + 2\|h_\varepsilon\| .
\end{align*}
Dividing by $n$, taking the limit as $n\to \infty$ and finally using that $\varepsilon$ is arbitrary yields the inequality $	 \limsup_{n\to\infty} \frac 1 n \|S_n f\| \leq \|\widetilde f\|_{*}$, which completes the proof.
\end{proof}

\begin{proof}[Proof of Theorem~\ref{thm:generalizedversion}]{}  Let $(f_n)_{n\geq 1} \subset V$ satisfy \eqref{eq:defasymostaddgen}. Then, for each $k\geq 1$, there exists $f^{(k)} \in W$ such that
\begin{equation}\label{eq:redefaa}
\limsup_{n\to\infty} \frac 1 n\|f_n-S_n f^{(k)}\| \leq \frac 1 k.	
\end{equation}
By \eqref{eq:norm1n} and \eqref{eq:redefaa},  we obtain that for all $k,\ell \geq 1$,
\begin{equation}\label{eq:cauchyseq}
\|\widetilde f^{(k)} -\widetilde  f^{(\ell)}\|_* = \lim_{n\to\infty} \frac 1 n\|S_n f^{(k)}-S_n f^{(\ell)}\| \leq \frac 1k + \frac 1\ell.	
\end{equation}
It follows that the sequence $(\widetilde f^{(k)})_{k\geq 1}$ is Cauchy in $(W/{\sim}, \|\cdot\|_*)$, and hence there exists $f\in W$ such that $\lim_{k\to\infty}\|\widetilde f^{(k)}-\widetilde f\|_* = 0$. We find, using again \eqref{eq:norm1n} and \eqref{eq:redefaa}, that for any $k\geq 1$,
\begin{align*}
\limsup_{n\to\infty} \frac 1 n\|f_n-S_n f\| &\leq \limsup_{n\to\infty} \frac 1 n\big(\|f_n-S_n f^{(k)}\|+\|S_n f^{(k)}-S_n f\|\big)	\leq \frac 1k + \|\widetilde f^{(k)}-\widetilde f\|_*.
\end{align*}
Taking the limit as $k\to\infty$ establishes \eqref{eq:asympaddnewgen}.  	\end{proof}

\begin{remark}\label{rem:afterproof} The vector $f\in W$ in \eqref{eq:asympaddnewgen} is not unique. The Cauchy sequence argument above yields an equivalence class $\widetilde f$, and indeed the set of all vectors $f\in W$ satisfying \eqref{eq:asympaddnewgen} forms an equivalence class under $\sim$, since by Lemma~\ref{lem:norme1nsn},
\begin{equation}\label{eq:equivrelSn}
	f\sim g \quad \text{ iff } \quad \lim_{n\to\infty} \frac 1n\|S_n f - S_n g\| = 0.
\end{equation}
Moreover, it is obvious from the proof that given any sequence  $( f^{(k)})_{k\geq 1} \subset W$ such that
$$\lim_{k\to\infty}\limsup_{n\to\infty} \frac 1 n\|f_n-S_n f^{(k)}\| =  0,$$
we have   $\lim_{k\to\infty}\widetilde f^{(k)} = \widetilde f$ in $(W/{\sim}, \|\cdot\|_*)$,
and by the discussion above, the limit $\widetilde f$ does not depend on the choice of $( f^{(k)})_{k\geq 1}$.
\end{remark}

\begin{remark}\label{rem:setupmainthm} In the setup of Theorem~\ref{th:maintheorem}, \ie when $V = W = C(X)$, $\|\cdot\| = \|\cdot\|_\infty$ and $Kf = f\circ T$, the following comments can be made:
\begin{enumerate}[leftmargin=2em]
\item The sequence $(n^{-1}\widetilde f_n)_{n\geq 1}$ converges to $\widetilde f$ in $C(X)/{\sim}$. Indeed,
 by \eqref{eq:1nSnequivf}, we have
\begin{equation*}
	\left\|\frac 1 n\widetilde f_n - \widetilde f\,\right\|_* \leq  \frac 1 n \|f_n - S_n f\|_\infty 
\end{equation*}
and the right-hand side converges to zero as $n\to\infty$ by \eqref{eq:asympaddnew}.
\item For all $f\in C(X)$, we have also $\|\widetilde f\|_{*} = \sup_{\mu\in \cP_T(X)} |\int f \d \mu|$, see for example \cite[Proposition 2.1]{phelps_israel_84}.
\item We have $\CL = {\rm cl}\{h - h\circ T: h\in C(X)\}$, where the closure is taken in $(C(X), \|\cdot\|_\infty)$. The elements of $\CL$ are sometimes called {\em weak coboundaries} \cite{BouschJenkinson2002}, and there are other equivalent definitions of $\CL$, see for example \cite[Proposition 2.13]{KatokRobinson2001}.  If $f\sim g$, then $f$ and $g$ are sometimes called {\em physically equivalent} in statistical mechanics  \cite[Section 4.6]{ruelle2004} or {\em weakly cohomologous}  \cite{BouschJenkinson2002}.
\item With \eqref{eq:equivrelSn} in mind, it is natural to generalize physical equivalence to non-additive sequences. We view two sequences $(f_n)_{n\geq 1}, (g_n)_{n\geq 1} \subset C(X)$ as physically equivalent if $\lim_{n\to\infty} \frac 1n\|f_n -  g_n\|_\infty = 0$ (see for example \cite[Remark A.6 (ii)]{feng_huang_2010} and \cite[Section 3.2]{bomfim_multifractal_2015} for similar considerations). This is indeed a generalization, since by \eqref{eq:equivrelSn}, for every $f,g\in C(X)$, the two additive sequences $(S_n f)_{n\geq 1}$ and $(S_n g)_{n\geq 1}$ are equivalent iff  $f\sim g$. In this language, \eqref{eq:asympaddnew} says that the sequence $(f_n)_{n\geq 1}$ is physically equivalent to the additive sequence $(S_n f)_{n\geq 1}$, and the precise implications of this are discussed in Section~\ref{sec:nonaddpot} below.
\end{enumerate}
\end{remark}

\section{Non-additive potential sequences}\label{sec:nonaddpot}

This section is intended for the reader less familiar with the non-additive thermodynamic formalism. We briefly review some basic concepts and results and explain the immediate consequences of Theorem~\ref{th:maintheorem} on the theory of almost additive and asymptotically additive sequences.

The first generalization of the thermodynamic formalism to non-additive sequences of functions was provided by Falconer in \cite{falconer_subadditive88}, where the main motivation was the dimensional analysis of non-conformal repellers.
Non-additive sequences of potentials typically arise from products of matrices of the kind
\begin{equation}\label{eq:fnprodmatrix}
	f_n(x) = h\big(M(x)\cdot M(Tx)\cdot M(T^2x) \cdots M(T^{n-1}(x))\big),
\end{equation}
where $M$ is a continuous matrix-valued function on $X$ (in many cases, $M$ is the Jacobian matrix of $T$ or its inverse) and $h$ is a real-valued function, typically $h(A) = \log \|A\|$ for some matrix norm $\|\cdot\|$. Except in very specific cases (for example if $M$ is the Jacobian of a conformal map $T$), the sequence $(f_n)_{n\geq 1}$ is not additive. However, depending on the properties of the matrices $M(x)$, the sequence can be almost additive or asymptotically additive.
We refer the reader to the numerous works cited below for many applications of non-additive sequences, and as of now we focus only on the ``abstract'' formalism of such sequences.

\subsection{Non-additive thermodynamic formalism}\label{sec:nonaddthermo}

We introduce the distance $d_n(x,y) = \max_{k\in \{0, 1, \dots, n-1\}}d(T^kx,T^ky)$, $n\geq 1$.
In this section, $\cF = (f_n)_{n\geq 1} \subset C(X)$ is any sequence of real-valued, continuous functions.
We now define the basic notions of the non-additive thermodynamic formalism, which are standard and direct generalizations of the corresponding concepts in the additive case. 

\begin{definition}The {\em topological pressure} of the sequence $\cF$ is defined as
	\begin{equation}\label{eq:defpressure}
	P_T(\cF) = \lim_{\varepsilon \downarrow 0}\limsup_{n\to\infty}\frac 1 n \log\Big( \sup_{E} \sum_{x\in E} e^{f_n(x)}\Big),
\end{equation}
where the supremum is taken over all $(n,\varepsilon)$-separated sets, \ie all sets $E\subset X$ such that for any $x,y\in E$, $x\neq y$, we have $d_n(x,y)\geq \varepsilon$ (see for example \cite{cao_thermodynamic_2008}).
\end{definition}
As in the additive case (see \cite[Chapter 9]{walters1982}, \cite[Section 20.2]{KH1995}), there are other possible definitions which, for the sequences that will be considered in the present paper, are equivalent to \eqref{eq:defpressure}. 
In the special case of shift spaces, one can, again as in the additive case, omit the parameter $\varepsilon$ and define $P_T(\cF)$ as  $\limsup_{n\to\infty}\frac 1 n \log\sum_{C} \sup_{x\in C} e^{f_n(x)}$, where the sum is over all cylinder sets $C$ of rank $n$.

\begin{definition}
Given $\mu \in \cP_T(X)$, we define the {\em average Lyapunov exponent} of $\cF$ by
\begin{equation}\label{eq:avgLyapexp}
	\cF_*(\mu) = \lim_{n\to\infty}\frac 1 n\int f_n \dd \mu.
\end{equation}
\end{definition}
The limit in \eqref{eq:avgLyapexp} exists and is finite for the sequences $\cF$ of interest in the present paper. We note that in the additive case, \ie if $f_n = S_n f$ for some $f\in C(X)$, we simply obtain $\cF_*(\mu) = \int f \d \mu$.

\begin{definition}
	The sequence $\cF$ satisfies the {\em variational principle} if\,\footnote{If $\cF_*(\mu)$ can take the value $-\infty$, which is not the case with almost and asymptotically additive sequences, a more detailed statement is required in order to remove some ambiguity, see for example \cite[Theorem 1.1]{cao_thermodynamic_2008}.}
\begin{equation}\label{eq:varprinciple}
	P_T(\cF) = \sup_{\mu\in  \cP_T(X)}\left(	\cF_*(\mu) + h_T(\mu)\right),
\end{equation}
where $h_T(\mu)$ is the Kolmogorov--Sinai entropy of $\mu$. 
Any measure $\mu\in  \cP_T(X)$ at which the supremum in \eqref{eq:varprinciple} is reached is called an {\em equilibrium state} (or {\em equilibrium measure}) for $\cF$. 
\end{definition}

The first variational principle for non-additive sequences was proved by Falconer in \cite{falconer_subadditive88}. There, the sequence $\cF$ is sub-additive, \ie $f_{n+m} \leq f_n + f_m\circ T^n$ for all $n, m\geq 1$. Results about sub-additive sequences are easily adapted to almost additive sequences, since for any almost additive sequence $\cF$, the sequence $(f_n + C)_{n\geq 1}$ is sub-additive (with $C$ as in \eqref{eq:defalmostadd}).\footnote{Our main result does not apply to general sub-additive sequences and we refer the reader to \cite{cao_thermodynamic_2008} and \cite[Chapters II.4 and III.7]{Bar2011} for a detailed exposition of the sub-additive thermodynamic formalism and its applications. The same comment applies to {\em asymptotically sub-additive sequences}, see for example \cite{feng_huang_2010,Yan_2013}.}

Another version of the variational principle was proved by Barreira \cite[Theorem 1.7]{barreira_1996} for sequences $\cF$ satisfying the following condition: there exists $f\in C(X)$ such that $\lim_{n\to\infty}\|f_{n+1} - f_{n} \circ T -  f\|_\infty = 0$.  Such sequences are asymptotically additive (see for example \cite[Proposition A.5 (iii)]{feng_huang_2010}), and, under the stronger condition $\sum_{n\geq 1}\|f_{n+1} - f_{n} \circ T -f\|_\infty <\infty$, they are almost additive.

The almost additive thermodynamic formalism was introduced by Barreira \cite{barreira-2006} and Mummert \cite{Mummert06, MummertThesis}. These works define almost additive sequences and prove, among other results, the variational principle under some additional assumptions that were later removed in \cite{cao_thermodynamic_2008}. See also \cite{FengLau2002,Feng2004} for earlier special cases of almost additive sequences of the form \eqref{eq:fnprodmatrix}.

The first discussion of the asymptotically additive thermodynamic formalism appears in the work \cite{feng_huang_2010} by Feng and Huang, which includes a proof of the variational principle. The definition of asymptotically additive sequences therein is slightly different, but equivalent to \eqref{eq:defasymostadd}. See  \cite[Proposition A.5]{feng_huang_2010} for some relations between the different definitions, and for a proof that almost additive sequences are asymptotically additive. 

Assume now that the sequence $\cF$ is almost or asymptotically additive, and write $f_n$ as in \eqref{eq:fsublinearcorr}.
The contribution of $u_n$ vanishes from \eqref{eq:defpressure} and \eqref{eq:avgLyapexp} in the limit $n\to\infty$, so that actually
\begin{equation*}
	P_T(\cF) = P_T(f), \qquad \cF_*(\mu)= \int f \d \mu,
\end{equation*}
where $P_T(f) = P_T((S_n f)_{n\geq 1 })$ is the standard topological pressure of $f$.
As a consequence, the variational principle \eqref{eq:varprinciple} boils down to the standard, well-known additive variational principle \cite{walters1982, KH1995, ruelle2004}:
\begin{equation*}
	P_T(f) = \sup_{\mu\in  \cP_T(X)}\left( \int f \d \mu + h_T(\mu)\right),
\end{equation*}
and the equilibrium states of $\cF$ are the same as the equilibrium states of $f$ in the usual sense.

\subsection{Gibbs and weak Gibbs measures}\label{sec:Gibbswg}

Many results about almost additive and asymptotically additive sequences involve Gibbs and weak Gibbs measures, which we now define.

\begin{definition}
	We say that a (not necessarily invariant) probability measure $\mu\in \cP(X)$ is a {\em Gibbs measure} for $\cF$ if for all $\varepsilon >0$ small enough, there exists $K_\varepsilon\geq 1$ such that 
\begin{equation*}
	K_\varepsilon^{-1} \leq \frac{\mu(\{y\in X: d_n(x,y) < \varepsilon\})}{e^{f_n(x)-nP_T(\cF)}}\leq K_\varepsilon, \qquad x\in X, n\geq 1.
\end{equation*}
We say that $\mu\in \cP(X)$ is a {\em weak Gibbs measure} for $\cF$ if there exists a family $(K_{n,\varepsilon})_{n\geq 1, \varepsilon > 0}\subset [1, \infty)$ such that
$\lim_{\varepsilon \downarrow 0} \limsup_{n\to\infty} n^{-1} \log K_{n,\varepsilon} = 0$
and
\begin{equation}\label{eq:defweakGibbs}
	K_{n,\varepsilon}^{-1} \leq \frac{\mu(\{y\in X: d_n(x,y) < \varepsilon\})}{e^{f_n(x)-nP_T(\cF)}}\leq K_{n,\varepsilon}, \qquad x\in X, n\geq 1, \varepsilon > 0.
\end{equation}
\end{definition}
In the additive case, \ie when $f_n = S_n f$ for some $f\in C(X)$, we simply say that $\mu$ is (weak) Gibbs with respect to $f$ (the notion of weak Gibbs measure with respect to a function $f$ was first introduced in \cite{Yuri_2003}). 
There are small variations in the definitions of (weak) Gibbs measures in the literature (see for example \cite[Definition 2.3]{varandas_weak_2015}). The discussion in the present paper remains unchanged for other variants of the above definition. In the special case of shift spaces (and, by extension, for systems which have a Markov partition), the common definition is that $\mu$ is a Gibbs measure for $\cF$ if there exists $K\geq 1$ such that
\begin{equation}\label{eq:Gibbsshift}
	K^{-1} \leq \frac{\mu(\cC_n(x))}{e^{f_n(x)-nP_T(\cF)}}\leq K, \qquad x\in X, n\geq 1,
\end{equation}
where $\cC_n(x)$ is the cylinder set of rank $n$ containing $x$, and we say that $\mu$ is a weak Gibbs measure if \eqref{eq:Gibbsshift} holds with $K$ replaced by $K_n$, where $\lim_{n\to\infty} n^{-1} \log K_n  =0$. See for example \cite{iommi_weak_2016}.

When $\cF$ is almost additive and satisfies the bounded variation condition (see Definition~\ref{def:boundedvar} below), the existence of a unique invariant Gibbs measure, which is also the unique equilibrium measure for $\cF$, was established in \cite{barreira-2006,Mummert06, MummertThesis}, see also \cite[Chapter IV.10]{Bar2011}. 

It is well known (see for example \cite[Lemma 4.4]{CJPS_topo}) that any invariant weak Gibbs measure for an asymptotically additive sequence $\cF$ is also an equilibrium measure. The converse is also true in statistical mechanics under some quite general conditions \cite{PS-2019}.

Now, if $\cF$ is asymptotically additive, we have again by Theorem~\ref{th:maintheorem} the decomposition \eqref{eq:fsublinearcorr}. The sublinear corrections $u_n$ can be absorbed into the constants $K_{n,\varepsilon}$ in \eqref{eq:defweakGibbs}, so that $\cF$ and $(S_n f)_{n\geq 1}$ share the same weak Gibbs measures (this is true with all variants of the definition). It follows that the set of all measures $\mu \in \cP(X)$ which are weak Gibbs with respect to asymptotically additive sequences of potentials is no larger than the set of weak Gibbs measures in the usual (additive) sense.

The same is, however, not true of Gibbs measures. In fact, all that can be said a priori is that any measure that is Gibbs with respect to $\cF$ is {\em weak} Gibbs with respect to  $f$. One can argue that this is enough for many applications, and that when working with asymptotically additive sequences of potentials, the notions of weak Gibbs and Gibbs measures are no different; in fact, on Markov shifts, any measure that is weak Gibbs with respect to an asymptotically additive sequence is also Gibbs for another, well-chosen, asymptotically additive sequence, see \cite{iommi_weak_2016}.
However, in the special case where $\cF$ is almost additive, the Gibbs property is genuinely stronger than the weak Gibbs property. This observation leads to an important open question which will be discussed in Section~\ref{sec:openquestionAlmostadd}.

\subsection{Dimension theory and multifractal analysis}

The dimensional analysis of the level sets of non-additive sequences is one of the main motivations for the development of the non-additive thermodynamic formalism. We limit ourselves to a very brief discussion here, and we refer the reader to the books \cite{PesinBook1997,Bar2011} for a detailed exposition of the dimension theory of dynamical systems and its many applications.

\begin{definition}
We define the {\em level sets} of $\cF$ (sometimes called the level sets of the Lyapunov exponent) by
\begin{equation}\label{eq:levelsets}
	E(\alpha) = \left\{x\in X: \lim_{n\to\infty} \frac 1n f_n(x) = \alpha\right\},\quad \alpha \in \rr,
\end{equation}
and we define the {\em irregular set} of $\cF$ as the set where the above limit does not exist.
\end{definition}
The topological entropy as well as various notions of dimension of $E(\alpha)$ and the irregular set have been abundantly studied. For example, let ${\cal E}(\alpha)$ be the topological entropy of the restriction of $T$ to $E(\alpha)$. Then, if $\cF$ is almost additive, Theorem 1 in \cite{Bar_Dout_2009} (see also \cite[Theorem 12.1.1]{Bar2011}) asserts that, under some conditions on the dynamical system, on $\cF$ and on $\alpha$, we have
\begin{equation}\label{eq:condvarprinc}
	{\cal E}(\alpha) = \sup\{h_T(\mu) : \mu \in \cP_T(X), \cF_*(\mu) = \alpha\} = \min\{P_T(q\cF) - q\alpha: q\in \rr\},
\end{equation}
where $q\cF = (q f_n)_{n\geq 1}$.

In fact, most of the attention has been devoted to the generalization
\begin{equation}\label{eq:levelsets2}
	E(\alpha) = \left\{x\in X: \lim_{n\to\infty} {\left(\frac{f_{1,n}(x)}{g_{1,n}(x)},\frac{f_{2,n}(x)}{g_{2,n}(x)}, \dots, \frac{f_{r,n}(x)}{g_{r,n}(x)}\right)}  = \alpha\right\},\quad \alpha \in \rr^r,
\end{equation}
where $r\geq 1$, and, for $1\leq i \leq r$, $\cF_i = (f_{i,n})_{n\geq 1}$ and $\GG_i = (g_{i,n})_{n\geq 1}$ are sequences of continuous functions.

For almost additive sequences, in the case $r = 2$ and $g_{i,n} = n$,  $i=1,2$, $n\geq 1$, a variational principle for the topological entropy of the level sets \eqref{eq:levelsets2} was established in \cite{barreira_gelfert_2006} for repellers under specific assumptions. In a more general case, the {\em $u$-dimension} (which generalizes the topological entropy) of the level sets \eqref{eq:levelsets2} and of the irregular set was studied in \cite{Bar_Dout_2009}. Further results about the level sets \eqref{eq:levelsets} and \eqref{eq:levelsets2} were obtained in \cite{BarralQu2012b} when the constant $\alpha$ is replaced by a continuous function $\alpha: X\to \rr$, see also \cite{BarralQu2012a}. A review was provided in \cite{barreira_rev_2010}, see also \cite{Bar2011}.

For related results and generalizations in the case of asymptotically additive sequences, we refer the reader to \cite{zhao_asymptotically_2011, BarreiraCaoWang2013,Cao_2013,Bar2011}. 

Now, Theorem~\ref{th:maintheorem} again reduces the analysis to the additive case. Writing $f_n$ as \eqref{eq:fsublinearcorr}, we indeed see that the sublinear corrections $u_n$ do not play any role in the definition \eqref{eq:levelsets} of $E(\alpha)$, so that $E(\alpha)$ is in fact the level set of the Birkhoff average $\lim_{n\to\infty} \frac 1n S_n f$. In the same way, $u_n$ vanishes from all the quantities in \eqref{eq:condvarprinc}, so that the relations \eqref{eq:condvarprinc} boil down to the corresponding ones in the additive case. This also applies to the irregular set, and the same considerations extend to the generalization \eqref{eq:levelsets2}, by applying Theorem~\ref{th:maintheorem} to each of the sequences $\cF_i$ and $\GG_i$. 

\begin{remark}Another comment about the limit $\lim_{n\to\infty} \frac 1n f_n$ for asymptotically additive sequences can be made: since $\lim_{n\to\infty} \frac 1n f_n = \lim_{n\to\infty} \frac 1n S_nf$ (when the limit exists), the conclusions of Birkhoff's ergodic theorem applied to $f$ immediately extend to the sequence $\cF$. This simplifies the proof of the version of Birkhoff's ergodic theorem for almost additive and asymptotically additive sequences given for example in \cite{barreira-2006,Mummert06,feng_huang_2010}.
\end{remark}

\subsection{Large deviations}

The asymptotic behavior of the sequence $(n^{-1} f_n)_{n\geq 1}$ can also be studied using the theory of large deviations. Here, we assume that  a sequence of probability measures $(\mu_n)_{n\geq 1} \subset \cP(X)$ is given, and we view $f_n$ as a random variable on $X$ with respect to $\mu_n$, $n\geq 1$.

\begin{definition}
The {\em Large Deviation Principle (LDP)} holds if there exists a lower semicontinuous function $I: \rr \mapsto [0,+\infty]$, called rate function, such that for any open set $O\subset \rr$ and any closed set $F\subset \rr$, we have
\begin{align*}
	\liminf_{n\to\infty} \frac 1 n \log \mu_n \left\{x \in X: \frac 1 n f_n(x) \in O\right\} &\geq -\inf_{x\in O} I(x),\\
	\limsup_{n\to\infty} \frac 1 n \log \mu_n\left\{x \in X: \frac 1 n f_n(x) \in F\right\} &\leq -\inf_{x\in F} I(x).
\end{align*}
\end{definition}
One may also consider the large deviations of the vector in \eqref{eq:levelsets2}, viewed as an $\rr^r$-valued random variable.

The LDP for almost additive and asymptotically additive sequences under some conditions was proved in \cite{bomfim_multifractal_2015,varandas_weak_2015,VarandasThaoGibbs2017}, see also \cite{BarralQu2012a}. The following result was proved in \cite{CJPS_topo}: under some mild conditions on the dynamical system, if $\cF$ is asymptotically additive, and if $\mu_n = \mu$ for all $n$, where $\mu$ is a weak Gibbs measure with respect to an asymptotically additive sequence  $\GG = (g_n)_{n\geq 1}$, then the LDP holds with a rate function $I$ given by 
\begin{equation}\label{eq:Iofx}
	I(x) = \sup_{\alpha \in \rr} (\alpha x - P_T(\GG + \alpha \cF) + P_T(\GG)), \quad x\in \rr,
\end{equation}
where the asymptotically additive sequence $\GG + \alpha \cF$ is defined in the obvious way.

Theorem~\ref{th:maintheorem} implies that $(n^{-1}f_n)_{n\geq 1}$ and $(n^{-1}S_nf)_{n\geq 1}$, seen as sequences of random variables with respect to the measures $\mu_n$, are {\em exponentially equivalent} (see \cite[Section 4.2.2]{DZ2000}), and as such they obey the exact same LDP (if any). Applying also Theorem~\ref{th:maintheorem} to the sequences $\GG$ and $\GG+\alpha \cF$, we obtain that the right-hand side of \eqref{eq:Iofx} can be written in terms of the corresponding additive potentials, so that the above LDP is no more general than previously known results in the additive case (see for example \cite{comman-2009,PS-2018}).

\begin{remark}
The short review about almost additive and asymptotically additive sequences above is far from complete. Numerous other works have studied or used almost additive sequences (see for example \cite{Ban_Chang_2011,Yayama_2011,barralfeng_2012,yayama_2011b,LiuGongZhou2016,Yayama_2016,barany2018domination,1078_0947_2019_4_2059}) and asymptotically additive sequences (see for example \cite{cheng_zhao_cao_2012,Yan_2013,zhao_conditional_2014,OjalaSuomalaWu2017,WangZhao2018,TianWangWang2019,sun2019ergodic}). 
 We also mention that recent works (for example \cite{IommiYayama2012,IommiYayama2014,LiZhang2017,Iommi2018,Ferreira2019}) considered such potential sequences on non-compact spaces, and an extension of our main result to such spaces will be proved in Section~\ref{sec:gennoncompactdisc}.	
\end{remark}

\section{Further consequences, extensions and open questions}\label{sec:corollariesetc}

We investigate in this section some less obvious consequences and extensions of Theorem~\ref{th:maintheorem}, and we discuss two open questions.

\subsection{Quasi-Bernoulli and weakly coupled measures}\label{sec:weakGibbsQB}

In this section, we give a corollary of Theorem~\ref{th:maintheorem} that applies to weakly coupled measures\footnote{Such measures were considered in \cite[Definition 8.2]{lewis_entropy_1995}, see also \cite{pfister_thermodynamical_2002,CJPS-2017}.} on shift spaces, and in particular, to quasi-Bernoulli measures. We let $X = \cA^{\nn}$ for some finite set $\cA$. We denote the points of $X$ by $x = (x_i)_{i\geq 1}$ and introduce the distance $d$ given by $d(x,y) = 0$ if $x=y$, and by $d(x,y) = 2^{-n(x,y)}$ where $n(x,y) = \min\{i\geq 1: x_i \neq y_i\}$ if $x\neq y$. We let $T$ be the left-shift, \ie $(T(x))_i = x_{i+1}$, $i\geq 1$.
Given a measure $\mu \in \cP(X)$, we write $\mu_n(a_1, \dots, a_n) = \mu(\{ x\in X: x_i = a_i, 1\leq i \leq n\})$, $a_1, \dots, a_n\in \cA$, $n\geq 1$. We now define two classes of measures.
\begin{definition}
A measure $\mu \in \cP(X)$ is {\em weakly coupled} if $\mu_1(a) > 0$ for all $a\in \cA$ and if there exists a sequence $(D_n)_{n\geq 1} \subset [1, \infty)$ satisfying $\lim_{n\to\infty} \frac 1 n \log D_n = 0$ such that
\begin{equation}\label{eq:weaklydecoupled}
D_n^{-1}\leq \frac{\mu_{n+m}	(a_1, \dots, a_{n+m})}{\mu_{n}	(a_1, \dots, a_{n}) \mu_{m}	(a_{n+1}, \dots, a_{n+m})} \leq D_n, \qquad m,n\geq 1,~ a_1, \dots, a_{n+m}\in \cA.
\end{equation}
A measure $\mu \in \cP(X)$ is {\em quasi-Bernoulli} if the above holds with $D_n = D$ for some $D\geq 1$ independent of $n$.
\end{definition}

Given a weakly coupled measure $\mu$, the sequence $\cF = (f_n)_{n\geq 1}$ given by
\begin{equation}\label{eq:deffnbernoulli}
	f_n(x) = \log \mu_n(x_1, \dots, x_n), \qquad n\geq 1,~x\in X
\end{equation}
satisfies the bound
\begin{equation}\label{eq:weaklyalmostadd}
	\|f_{n+m}-f_n - f_m\circ T^n\|_\infty\leq \log D_n, \quad n,m\geq 1.
\end{equation}
In particular, if $\mu$ is quasi-Bernoulli, then $\cF$ is almost additive with $C = \log D$. It turns out that \eqref{eq:weaklyalmostadd} with $\lim_{n\to\infty} \frac 1 n \log D_n = 0$ is enough to ensure that $\cF$ is asymptotically additive (a variation of the standard argument that applies to this case is provided in \cite[Theorem 6.1 (3)]{CJPS_topo}). 
We thus obtain the following immediate corollary of Theorem~\ref{th:maintheorem}:
\begin{corollary}\label{cor:quasiB}
	Let $\mu$ be a weakly coupled measure (or a quasi-Bernoulli measure). Then $\mu$ is weak Gibbs with respect to an additive potential, \ie there exists $f\in C(X)$ and a sequence $(K_n)_{n\geq 1}\subset [1,\infty)$ such that $\lim_{n\to\infty} n^{-1} \log K_n = 0$ and 
\begin{equation}\label{eq:weakGibbsBer}
K_n^{-1} \leq \frac{\mu_n(x_1, \dots, x_n)}{e^{S_n f(x)}} \leq K_n, \quad n\geq 1,~x\in X.
\end{equation}
\end{corollary}
\begin{remark} Note that the function $f$ in  \eqref{eq:weakGibbsBer} has vanishing topological pressure, since the topological pressure of the sequence \eqref{eq:deffnbernoulli} is easily seen to be zero (see also \cite{iommi_weak_2016}).
\end{remark}

Corollary~\ref{cor:quasiB} gives a partial answer to a question raised in \cite[p.~1421]{BarralMensi2007} and also by Pablo Shmerkin:

{\bf Open question:} Is there, for every quasi-Bernoulli measure  $\mu$,  a potential $f\in C(X)$ with respect to which $\mu$ is a Gibbs measure, \ie such that \eqref{eq:weakGibbsBer} holds with $K_n$ replaced by a constant $K\geq 1$ independent of $n$?

We make the following comments about this question:
\begin{enumerate}[leftmargin=2em]
	\item The converse is true: if $\mu$ is Gibbs with respect to some $f\in C(X)$, then $\mu$ is quasi-Bernoulli. 
\item Corollary~\ref{cor:quasiB} provides a partial answer by guaranteeing that $\mu$ is at least {\em weak} Gibbs for some $f\in C(X)$. 
\item An example of quasi-Bernoulli measure that is not Gibbs with respect to any {\em Hölder}-continuous potential is given in \cite[Example 2.10 (2)]{barany2018domination}, but this measure may still be Gibbs with respect to a continuous potential.
\item If $\mu$ is weakly coupled (instead of quasi-Bernoulli), then $\mu$ is not, in general, Gibbs for some $f\in C(X)$. Indeed, if $\mu$ is invariant, the Gibbs property implies that $\mu$ is mixing \cite[Proposition 20.3.6]{KH1995}, whereas many invariant measures in statistical mechanics are weakly coupled without even being ergodic.
\item The question is a special case of the open problem stated in the next section.
\end{enumerate}

\subsection{An open question about almost additive sequences and Gibbs measures}\label{sec:openquestionAlmostadd}

We discuss here some specific properties of almost additive sequences and the associated Gibbs measures.
The following definition (see for example equation (2.1) in  \cite{bomfim_multifractal_2015}) is a direct generalization of the condition introduced by Bowen in the additive case \cite{bowen_7475}.

\begin{definition}\label{def:boundedvar}
A sequence of functions $(f_n)_{n\geq 1}\subset C(X)$ satisfies the {\em bounded variation condition} if there exists $\varepsilon > 0$ such that 
\begin{equation}\label{eq:boundedmildvar}
\sup_{n\geq 1}\sup_{\substack{x,y\in X, \\d_n(x,y)<\varepsilon}} 
\bigl|f_n(x)-f_n(y)\bigr| < \infty  .	
\end{equation}
\end{definition}
In the case of shift spaces, using the notation of Section~\ref{sec:weakGibbsQB}, this condition becomes
\begin{equation}\label{eq:bvarshift}
\sup_{n\geq 1}\sup_{\substack{x,y\in X, \\ x_i = y_i, 1\leq i \leq n}} 
 \bigl|f_n(x)-f_n(y)\bigr| < \infty .	
\end{equation}
As mentioned in Section~\ref{sec:Gibbswg}, when $\cF = (f_n)_{n\geq 1}$ is almost additive and satisfies the bounded variation condition, the existence of a unique invariant Gibbs measure, which is also the unique equilibrium measure for $\cF$, was established for some dynamical systems in \cite{barreira-2006,Mummert06, MummertThesis,Bar2011}. (The same is not true with asymptotically additive sequences with bounded variations, see Remark~\ref{rem:notuniqueAsa}.)

With $\cF$ as above, there exists by Theorem~\ref{th:maintheorem} a function $f\in C(X)$ such that \eqref{eq:asympaddnew} holds. However, even though the bounded variation condition is satisfied by $\cF$, it is not, in general, satisfied by the sequence $(S_n f)_{n\geq 1}$. Thus, the uniqueness argument in \cite{barreira-2006,Mummert06, MummertThesis,Bar2011} cannot be replaced by the corresponding standard argument for the additive potential $f$ (see for example \cite{bowen_7475} or \cite[Theorem 20.3.7]{KH1995}), since this argument requires $(S_n f)_{n\geq 1}$ to have bounded variations.
In the same way, the unique equilibrium measure for $\cF$ is Gibbs with respect to $\cF$, but it is in general only weak Gibbs with respect to $f$, as discussed previously.

In fact, weak Gibbs measures are a property of the equivalence class $\widetilde f$ (recall Remark~\ref{rem:setupmainthm}): if $f\sim f'$, then $\mu$ is  weak Gibbs with respect to $f$ iff it is weak Gibbs with respect to $f'$. On the contrary, Gibbs measures and the bounded variation condition may depend on the actual function within an equivalence class.

We are thus naturally brought to the following question:

{\bf Open question:} Is there, given any almost additive sequence $(f_n)_{n\geq 1}$ with bounded variations, a function $f\in C(X)$ such that
\begin{equation}\label{eq:strongerequiv}
\sup_{n\geq 1} \|f_n-S_n f\|_\infty < \infty \qquad ?	
\end{equation}
The condition \eqref{eq:strongerequiv} is stronger than \eqref{eq:asympaddnew}.
If  \eqref{eq:strongerequiv} holds, it is immediate that $(S_n f)_{n\geq 1}$ also has bounded variations and that the unique invariant Gibbs measure for $\cF$ is also the unique invariant Gibbs measure for $f$.

As a corollary, a positive answer to the question above would imply that every quasi-Bernoulli measure (see Section~\ref{sec:weakGibbsQB}) is actually a Gibbs measure for some $f\in C(X)$, since, in the case of quasi-Bernoulli measures, the sequence $\cF$ given by \eqref{eq:deffnbernoulli} is almost additive and has bounded variations (the left-hand side of \eqref{eq:bvarshift} is actually zero).

\begin{remark}\label{rem:notuniqueAsa}
In the case of asymptotically additive sequences (instead of almost additive ones), the bounded variation condition does not guarantee the uniqueness of the equilibrium state, even in the simplest setups. For example, let $X$ be the full shift $\cA^\nn$ as in Section~\ref{sec:weakGibbsQB} and let $f\in C(X)$ be a function which admits (at least) two distinct invariant weak Gibbs measures $\mu, \mu'$ (such examples are well known in statistical mechanics: apply the main result of \cite{PS-2019} to any interaction admitting several equilibrium states). Let then $\cF$ be as in \eqref{eq:deffnbernoulli}. Here again, $\cF$ has bounded variations. Moreover, we have $\lim_{n\to\infty}\frac 1n\|f_n - S_n f \|_\infty = 0$, and thus $\cF$ is asymptotically additive (this was also observed in \cite{iommi_weak_2016}). However, $\cF$  admits $\mu$ as a Gibbs measure, and also $\mu'$ as a weak Gibbs measure; in particular, both $\mu$ and $\mu'$ are equilibrium states for $\cF$.
\end{remark}

\subsection{Non-compact spaces and discontinuous potentials}\label{sec:gennoncompactdisc}

For simplicity, we have up to now limited ourselves to continuous potentials and continuous sequences of potentials on compact metric spaces, which is the most common setup (although the non-compact case has been considered for example in \cite{IommiYayama2012,IommiYayama2014,LiZhang2017,Iommi2018,Ferreira2019}). In this section, we let $X$ be any topological space.
We give a generalization of Theorem~\ref{th:maintheorem} which applies to such $X$, and without any regularity condition on the functions $f_n$.\footnote{In practice, it is most common to require the functions $f_n$ to be at least measurable, but this is not necessary for Theorem~\ref{th:thmdiscontinuousre}.}  The map $T:X\to X$ is still assumed to be continuous.
We denote by $B(X)$ the space of bounded functions, and by $C_b(X)$ the space of bounded continuous functions on $X$. Both spaces are Banach with respect to the norm $\|\cdot \|_\infty$.

\begin{theorem}\label{th:thmdiscontinuousre}
Let $X$ be a topological space, and let $(f_n)_{n\geq 1}$ be a sequence of real-valued functions satisfying\footnote{The quantity $\|f_n - S_n f\|_\infty$ in \eqref{eq:defasymostaddrere} is allowed to be infinite, since both $f$ and $f_n$ may be unbounded.}
\begin{equation}\label{eq:defasymostaddrere}
	\inf_{f\in C(X)}\limsup_{n\to\infty} \frac 1 n \|f_n - S_n f\|_\infty = 0.
\end{equation}
Then, there exists $f\in C(X)$ such that \eqref{eq:asympaddnew} holds.\end{theorem}

\begin{proof}{} We will apply Theorem~\ref{thm:generalizedversion} with $V = B(X), W = C_b(X)$, $\|\cdot\| = \|\cdot\|_\infty$, and again $Kf = f\circ T$. 

We first reduce the problem to the case where $f_n\in B(X)$ for all $n\geq 1$. By \eqref{eq:defasymostaddrere}, there exists $g\in C(X)$ and $n_0$ such that $\|f_n -S_n g\|_\infty < n$ for all $n\geq n_0$. Thus, the sequence $(f_n')_{n\geq 1}$ defined by $f_n' \equiv 0$ if $n<n_0$ and $f_n'  = f_n - S_n g$ otherwise is such that $\|f_n'\|_\infty < \infty$ for all $n$, and by replacing $f$ with $f+g$, we obtain that the sequence $(f_n')_{n\geq 1}$ also satisfies  \eqref{eq:defasymostaddrere}.
Moreover, \eqref{eq:asympaddnew} holds for $(f_n')_{n\geq 1}$ iff it holds for $(f_n)_{n\geq 1}$ (again by replacing $f$ with $f+g$). So as of now, we assume without loss of generality that $f_n \in B(X)$ for all $n$.

In order to apply Theorem~\ref{thm:generalizedversion}, it remains to show that we can replace the infimum in \eqref{eq:defasymostaddrere} with an infimum over the set $C_b(X)$, \ie that 
\begin{equation}\label{eq:defasymostaddrererere}
	\inf_{f\in C_b(X)}\limsup_{n\to\infty} \frac 1 n \|f_n - S_n f\|_\infty = 0.
\end{equation}
To prove \eqref{eq:defasymostaddrererere}, let $f\in C(X)\setminus C_b(X)$. Then  $\|S_n f\|_\infty = \infty$  for infinitely many values of $n$ (indeed, if $\|S_n f\|_\infty < \infty$ and $\|f\|_\infty=\infty$, then $\|S_{n+1} f \|_\infty = \|f + (S_n f)\circ T\|_\infty = \infty$). Since $\|f_n\|_\infty < \infty$ for all $n$, we find $\limsup_{n\to\infty} \frac 1 n \|f_n - S_n f\|_\infty = \infty$. Thus, the functions in $C(X)\setminus C_b(X)$ can be discarded in the infimum in \eqref{eq:defasymostaddrere}, so that \eqref{eq:defasymostaddrererere} indeed holds. This completes the proof.
\end{proof}

In the setup of Theorem~\ref{th:thmdiscontinuousre}, the decomposition \eqref{eq:fsublinearcorr} remains valid with some (not necessarily continuous) functions $u_n$ satisfying $\lim_{n\to\infty} n^{-1}\|u_n \|_\infty = 0$.

The condition \eqref{eq:defasymostaddrere} generalizes the definition of asymptotically additive sequences to non-compact spaces and possibly discontinuous functions $f_n$. We now show that, also in the non-compact case, almost additive sequences are asymptotically additive. The standard proof  (see \cite[Proposition A.5 (ii)]{feng_huang_2010} or \cite[Proposition 2.1]{zhao_asymptotically_2011}) requires each function $f_n$ to be bounded, which may not be the case on unbounded spaces. We show here how the argument can be adapted.
\begin{proposition}\label{prop:almostaddunbounded}
Let $X$ be a topological space, and let $(f_n)_{n\geq 1}\subset C(X)$ be a sequence of continuous functions satisfying  \eqref{eq:defalmostadd} for some $C\geq 0$. Then the sequence $(f_n)_{n\geq 1}$ satisfies \eqref{eq:defasymostaddrere}. In particular, the conclusion of Theorem~\ref{th:thmdiscontinuousre} holds.
\end{proposition}
\begin{proof}{} By \eqref{eq:defalmostadd}, we have $\|f_n - S_n f_1 \|_\infty \leq (n-1)C$ for all $n$. The sequence $(f_n')_{n\geq 1} = (f_n - S_n f_1)_{n\geq 1}$ still satisfies \eqref{eq:defalmostadd} with the same constant $C$, and since $\|f_n'\|_\infty < \infty$ for all $n$, the standard proof (\cite[Proposition A.5 (ii)]{feng_huang_2010} or \cite[Proposition 2.1]{zhao_asymptotically_2011}) implies that $(f_n')_{n\geq 1}$ satisfies \eqref{eq:defasymostaddrere}. By replacing $f$ with $f+f_1$ in \eqref{eq:defasymostaddrere}, we obtain the conclusion of the proposition. \end{proof}

\subsection{Moderate variation condition}\label{sec:moderatevar}

We now introduce the {\em moderate variation condition} (also sometimes called {\em tempered} or {\em mild} variation condition), which is weaker than the bounded variation condition (Definition~\ref{def:boundedvar}). In the remainder of this section, we assume for simplicity that $(X,d)$ is again a compact metric space.
\begin{definition}
A sequence of functions $(f_n)_{n\geq 1}$ satisfies the {\em moderate variation condition} if\,\footnote{There are minor variations of this condition in the literature, compare for example with equation (15) in \cite{barreira_gelfert_2006} and equation (2.2) in \cite{bomfim_multifractal_2015}.}
\begin{equation}\label{eq:mildvar}
\lim_{\varepsilon \downarrow 0}\limsup_{n\to\infty}\sup_{\substack{x,y\in X, \\d_n(x,y)<\varepsilon}} 
\frac 1 n\bigl|f_n(x)-f_n(y)\bigr| = 0.	
\end{equation}
\end{definition}
As with Definition~\ref{def:boundedvar}, one can simplify \eqref{eq:mildvar} in the case of shift spaces.
It is easy to realize that in the continuous, additive case, \ie if $f_n = S_n f$ for some $f\in C(X)$, then \eqref{eq:mildvar} holds. 

This moderate variation condition is often required for non-additive sequences, and it was observed in \cite[Lemma 2.1 and Remark 1]{zhao_asymptotically_2011} that this condition is actually redundant in the asymptotically additive case, since any asymptotically additive sequence automatically satisfies \eqref{eq:mildvar}. This is easily seen using Theorem~\ref{th:maintheorem} or Theorem~\ref{th:thmdiscontinuousre}. 

We present now a version of Proposition~\ref{prop:almostaddunbounded} where the continuity condition on $f_n$ is replaced by the moderate variation condition.
\begin{proposition}\label{prop:almostaddunboundedencore}
Assume $(X,d)$ is a compact metric space, and let $(f_n)_{n\geq 1}$ be a sequence of real-valued functions satisfying the moderate variation condition and   \eqref{eq:defalmostadd} for some $C\geq 0$. Then the sequence $(f_n)_{n\geq 1}$ satisfies \eqref{eq:defasymostaddrere}. In particular, the conclusion of Theorem~\ref{th:thmdiscontinuousre} holds.
\end{proposition}
\begin{proof}{} The proof that $(f_n)_{n\geq 1}$ satisfies \eqref{eq:defasymostaddrere} is given in \cite[Theorem 6.1 (4)]{CJPS_topo} and we only outline it here. The first step is to show, using the same argument as \cite[Proposition A.5 (ii)]{feng_huang_2010} or \cite[Proposition 2.1]{zhao_asymptotically_2011}, that for any $k\geq 1$,
\begin{equation}\label{eq:limnSnSn}
	\lim_{k\to\infty}\limsup_{n\to\infty}\frac 1n\left\| f_n - S_n \left(\frac {f_k} k\right) \right\|_\infty = 0.
\end{equation}
Now by the condition \eqref{eq:mildvar}, we can find, for each $k\geq 1$, a regularization $f^{(k)}\in C(X)$ of the function $\frac {f_k} k$ such that
\begin{equation*}
	\lim_{k\to\infty}\left\|\frac {f_k}k - f^{(k)} \right\|_\infty = 0.
\end{equation*}
By combining this and \eqref{eq:limnSnSn}, we obtain 
\begin{equation*}
	\lim_{k\to\infty}\limsup_{n\to\infty}\frac 1n\left\|f_n - S_n f^{(k)} \right\|_\infty = 0,
\end{equation*}
which is what we needed to show.	
\end{proof}

We conclude with a simple corollary of Proposition~\ref{prop:almostaddunboundedencore}, which addresses the case of additive, discontinuous potentials.
\begin{corollary}\label{cor:addModVar}
Let $\widehat f : X\to \rr$ be such that $(S_n \widehat f)_{n\geq 1}$ satisfies the moderate variation condition, \ie 
\begin{equation}\label{eq:mildvaraddB}
\lim_{\varepsilon \downarrow 0}\limsup_{n\to\infty}\sup_{\substack{x,y\in X, \\d_n(x,y)<\varepsilon}} 
\frac 1 n\bigl|S_n \widehat f(x)-S_n \widehat f(y)\bigr| = 0.	
\end{equation}	
Then, there exists $f \in C(X)$ such that
\begin{equation}\label{eq:addadd}
\lim_{n\to\infty}\frac 1n\|S_n \widehat f - S_n f \|_\infty = 0.
\end{equation}
\end{corollary}

Thus, by Corollary~\ref{cor:addModVar}, any potential satisfying \eqref{eq:mildvaraddB} is equivalent to a continuous one in the sense discussed in Section~\ref{sec:nonaddpot}.

\addcontentsline{toc}{section}{References}

\def\polhk#1{\setbox0=\hbox{#1}{\ooalign{\hidewidth
  \lower1.5ex\hbox{`}\hidewidth\crcr\unhbox0}}} \def\pre{{Phys. Rev. E\ }}


\begin{thebibliography}{10}

\bibitem{Ban_Chang_2011}
J.-C. Ban and C.-H. Chang.
\newblock Factor map, diamond and density of pressure functions.
\newblock {\em Proc. Amer. Math. Soc.\/} {\bf 139} (2011), 3985--3997.

\bibitem{barany2018domination}
B.~B{\'a}r{\'a}ny, A.~K{\"a}enm{\"a}ki, and I.~D. Morris.
\newblock Domination, almost additivity, and thermodynamical formalism for
  planar matrix cocycles.
\newblock {\em arXiv:1802.01916\/}.

\bibitem{barralfeng_2012}
J.~Barral and D.-J. Feng.
\newblock Weighted thermodynamic formalism on subshifts and applications.
\newblock {\em Asian J. Math.\/} {\bf 16} (2012), 319--352.

\bibitem{BarralMensi2007}
J.~Barral and M.~Mensi.
\newblock Gibbs measures on self-affine {S}ierpi\'{n}ski carpets and their
  singularity spectrum.
\newblock {\em Ergod. Theory Dyn. Syst.\/} {\bf 27} (2007), 1419--1443.

\bibitem{BarralQu2012b}
J.~Barral and Y.-H. Qu.
\newblock Localized asymptotic behavior for almost additive potentials.
\newblock {\em Discrete Contin. Dyn. Syst.\/} {\bf 32} (2012), 717--751.

\bibitem{BarralQu2012a}
J.~Barral and Y.-H. Qu.
\newblock On the higher-dimensional multifractal analysis.
\newblock {\em Discrete Contin. Dyn. Syst.\/} {\bf 32} (2012), 1977--1995.

\bibitem{barreira_1996}
L.~Barreira.
\newblock A non-additive thermodynamic formalism and applications to dimension
  theory of hyperbolic dynamical systems.
\newblock {\em Ergodic Theory Dynam. Systems\/} {\bf 16} (1996), 871--927.

\bibitem{barreira-2006}
L.~Barreira.
\newblock Nonadditive thermodynamic formalism: equilibrium and {G}ibbs
  measures.
\newblock {\em Discrete Contin. Dyn. Syst.\/} {\bf 16} (2006), 279--305.

\bibitem{barreira_rev_2010}
L.~Barreira.
\newblock Almost additive thermodynamic formalism: some recent developments.
\newblock {\em Rev. Math. Phys.\/} {\bf 22} (2010), 1147--1179.

\bibitem{Bar2011}
L.~Barreira.
\newblock {\em {Thermodynamic Formalism and Applications to Dimension
  Theory}\/} (Birkh\"auser/Springer Basel AG, Basel, 2011).

\bibitem{BarreiraCaoWang2013}
L.~Barreira, Y.~Cao, and J.~Wang.
\newblock Multifractal analysis of asymptotically additive sequences.
\newblock {\em J. Stat. Phys.\/} {\bf 153} (2013), 888--910.

\bibitem{Bar_Dout_2009}
L.~Barreira and P.~Doutor.
\newblock Almost additive multifractal analysis.
\newblock {\em J. Math. Pures Appl. (9)\/} {\bf 92} (2009), 1--17.

\bibitem{barreira_gelfert_2006}
L.~Barreira and K.~Gelfert.
\newblock Multifractal analysis for {L}yapunov exponents on nonconformal
  repellers.
\newblock {\em Commun. Math. Phys.\/} {\bf 267} (2006), 393--418.

\bibitem{bomfim_multifractal_2015}
T.~Bomfim and P.~Varandas.
\newblock Multifractal analysis of the irregular set for almost-additive
  sequences via large deviations.
\newblock {\em Nonlinearity\/} {\bf 28} (2015), 3563--3585.

\bibitem{BouschJenkinson2002}
T.~Bousch and O.~Jenkinson.
\newblock Cohomology classes of dynamically non-negative {$C^k$} functions.
\newblock {\em Invent. Math.\/} {\bf 148} (2002), 207--217.

\bibitem{bowen_7475}
R.~Bowen.
\newblock Some systems with unique equilibrium states.
\newblock {\em Math. Systems Theory\/} {\bf 8} (1974/75), 193--202.

\bibitem{Cao_2013}
Y.~Cao.
\newblock Dimension spectrum of asymptotically additive potentials for {$C^1$}
  average conformal repellers.
\newblock {\em Nonlinearity\/} {\bf 26} (2013), 2441--2468.

\bibitem{cao_thermodynamic_2008}
Y.~Cao, D.-J. Feng, and W.~Huang.
\newblock The thermodynamic formalism for sub-additive potentials.
\newblock {\em Discrete Contin. Dyn. Syst.\/} {\bf 20} (2008), 639--657.

\bibitem{cheng_zhao_cao_2012}
W.-C. Cheng, Y.~Zhao, and Y.~Cao.
\newblock Pressures for asymptotically sub-additive potentials under a mistake
  function.
\newblock {\em Discrete Contin. Dyn. Syst. Ser. A\/} {\bf 32} (2012), 487--497.

\bibitem{comman-2009}
H.~Comman.
\newblock Strengthened large deviations for rational map and full-shifts, with
  unified proof.
\newblock {\em Nonlinearity\/} {\bf 22} (2009), 1413--1429.

\bibitem{CJPS_topo}
N.~Cuneo, V.~{Jak\v si\'c}, C.-A. Pillet, and A.~Shirikyan.
\newblock Fluctuation theorem and thermodynamic formalism.
\newblock {\em arXiv:1712.05167\/}.

\bibitem{CJPS-2017}
N.~Cuneo, V.~{Jak\v si\'c}, C.-A. Pillet, and A.~Shirikyan.
\newblock Large deviations and fluctuation theorem for selectively decoupled
  measures on shift spaces.
\newblock {\em Rev. Math. Phys.\/} {\bf 31} (2019), 1950036--1--54.

\bibitem{DZ2000}
A.~Dembo and O.~Zeitouni.
\newblock {\em {Large Deviations Techniques and Applications}\/}
  (Springer--Verlag, Berlin, 2000).

\bibitem{fabian2011banach}
M.~Fabian, P.~Habala, P.~H{\'a}jek, V.~Montesinos, and V.~Zizler.
\newblock {\em {B}anach {S}pace {T}heory: {T}he {B}asis for {L}inear and
  {N}onlinear {A}nalysis\/}.
\newblock CMS Books in Mathematics (Springer, 2011).

\bibitem{falconer_subadditive88}
K.~J. Falconer.
\newblock A subadditive thermodynamic formalism for mixing repellers.
\newblock {\em J. Phys. A Math. Gen.\/} {\bf 21} (1988), L737--L742.

\bibitem{Feng2004}
D.-J. Feng.
\newblock The variational principle for products of non-negative matrices.
\newblock {\em Nonlinearity\/} {\bf 17} (2004), 447--457.

\bibitem{feng_huang_2010}
D.-J. Feng and W.~Huang.
\newblock Lyapunov spectrum of asymptotically sub-additive potentials.
\newblock {\em Commun. Math. Phys.\/} {\bf 297} (2010), 1--43.

\bibitem{FengLau2002}
D.-J. Feng and K.-S. Lau.
\newblock The pressure function for products of non-negative matrices.
\newblock {\em Math. Res. Lett.\/} {\bf 9} (2002), 363--378.

\bibitem{Ferreira2019}
G.~Ferreira.
\newblock The asymptotically additive topological pressure: variational
  principle for non-compact and intersection of irregular sets.
\newblock {\em Dynamical Systems\/} {\bf 34} (2019), 484--503.

\bibitem{Iommi2018}
G.~Iommi, C.~Lacalle, and Y.~Yayama.
\newblock Hidden {G}ibbs measures on shift spaces over countable alphabets.
\newblock {\em Stochastics and Dynamics\/}  (2019), online ready.

\bibitem{IommiYayama2012}
G.~Iommi and Y.~Yayama.
\newblock Almost-additive thermodynamic formalism for countable {M}arkov
  shifts.
\newblock {\em Nonlinearity\/} {\bf 25} (2012), 165--191.

\bibitem{IommiYayama2014}
G.~Iommi and Y.~Yayama.
\newblock Zero temperature limits of {G}ibbs states for almost-additive
  potentials.
\newblock {\em J. Stat. Phys.\/} {\bf 155} (2014), 23--46.

\bibitem{iommi_weak_2016}
G.~Iommi and Y.~Yayama.
\newblock Weak {G}ibbs measures as {G}ibbs measures for asymptotically additive
  sequences.
\newblock {\em Proc. Amer. Math. Soc.\/} {\bf 145} (2017), 1599--1614.

\bibitem{phelps_israel_84}
R.~B. Israel and R.~R. Phelps.
\newblock Some convexity questions arising in statistical mechanics.
\newblock {\em Math. Scand.\/} {\bf 54} (1984), 133--156.

\bibitem{KH1995}
A.~Katok and B.~Hasselblatt.
\newblock {\em Introduction to the {M}odern {T}heory of {D}ynamical
  {S}ystems\/} (Cambridge University Press, Cambridge, 1995).

\bibitem{KatokRobinson2001}
A.~Katok and E.~A. Robinson~Jr.
\newblock Cocycles, cohomology and combinatorial constructions in ergodic
  theory.
\newblock In: {\em Smooth ergodic theory and its applications ({S}eattle, {WA},
  1999)\/}, volume~69 of {\em Proc. Sympos. Pure Math.\/} (Amer. Math. Soc.,
  Providence, RI, 2001), pp. 107--173.

\bibitem{lewis_entropy_1995}
J.~T. Lewis, C.-E. Pfister, and W.~G. Sullivan.
\newblock Entropy, concentration of probability and conditional limit theorems.
\newblock {\em Markov Process. Related Fields\/} {\bf 1} (1995), 319--386.

\bibitem{LiZhang2017}
Z.~Li and X.~Zhang.
\newblock Almost additive topological pressure of proper systems.
\newblock {\em J. Dyn. Control Syst.\/} {\bf 23} (2017), 839--852.

\bibitem{LiuGongZhou2016}
L.~Liu, H.~Gong, and X.~Zhou.
\newblock Topological pressure dimension for almost additive potentials.
\newblock {\em Dyn. Syst.\/} {\bf 31} (2016), 357--374.

\bibitem{Mummert06}
A.~Mummert.
\newblock The thermodynamic formalism for almost-additive sequences.
\newblock {\em Discrete Contin. Dyn. Syst.\/} {\bf 16} (2006), 435--454.

\bibitem{MummertThesis}
A.~Mummert.
\newblock Thermodynamic formalism for nonuniformly hyperbolic dynamical
  systems.
\newblock Ph.D. thesis, The Pennsylvania State University (2006).

\bibitem{OjalaSuomalaWu2017}
T.~Ojala, V.~Suomala, and M.~Wu.
\newblock Random cutout sets with spatially inhomogeneous intensities.
\newblock {\em Israel J. Math.\/} {\bf 220} (2017), 899--925.

\bibitem{PesinBook1997}
Y.~B. Pesin.
\newblock {\em Dimension {T}heory in {D}ynamical {S}ystems\/}.
\newblock Chicago Lectures in Mathematics (University of Chicago Press,
  Chicago, 1997).

\bibitem{pfister_thermodynamical_2002}
C.-E. Pfister.
\newblock Thermodynamical aspects of classical lattice systems.
\newblock In: {\em In and out of equilibrium ({M}ambucaba, 2000)\/}, volume~51
  of {\em Progr. Probab.\/} (Birkh\"auser, Boston, 2002), pp. 393--472.

\bibitem{PS-2018}
C.-E. Pfister and W.~G. Sullivan.
\newblock Weak {G}ibbs measures and large deviations.
\newblock {\em Nonlinearity\/} {\bf 31} (2017), 49--53.

\bibitem{PS-2019}
C.-E. Pfister and W.~G. Sullivan.
\newblock Weak {Gibbs} and equilibrium measures for shift spaces.
\newblock {\em arXiv:1901.11488\/}.

\bibitem{ruelle2004}
D.~Ruelle.
\newblock {\em Thermodynamic {F}ormalism\/} (Cambridge University Press,
  Cambridge, 2004).

\bibitem{sun2019ergodic}
P.~Sun.
\newblock Ergodic measures of intermediate entropies for dynamical systems with
  approximate product property.
\newblock {\em arXiv:1906.09862\/}.

\bibitem{TianWangWang2019}
X.~Tian, S.~Wang, and X.~Wang.
\newblock Intermediate {L}yapunov exponents for systems with periodic orbit
  gluing property.
\newblock {\em Discrete Contin. Dyn. Syst.\/} {\bf 39} (2019), 1019--1032.

\bibitem{van_enter_regularity_1993}
A.~C.~D. {van Enter}, R.~Fern\'andez, and A.~D. Sokal.
\newblock Regularity properties and pathologies of position-space
  renormalization-group transformations: scope and limitations of {G}ibbsian
  theory.
\newblock {\em J. Stat. Phys.\/} {\bf 72} (1993), 879--1167.

\bibitem{varandas_weak_2015}
P.~Varandas and Y.~Zhao.
\newblock Weak specification properties and large deviations for non-additive
  potentials.
\newblock {\em Ergod. Theory Dyn. Syst.\/} {\bf 35} (2015), 968--993.

\bibitem{VarandasThaoGibbs2017}
P.~Varandas and Y.~Zhao.
\newblock Weak {G}ibbs measures: speed of convergence to entropy, topological
  and geometrical aspects.
\newblock {\em Ergod. Theory Dyn. Syst.\/} {\bf 37} (2017), 2313--2336.

\bibitem{walters1982}
P.~Walters.
\newblock {\em {An Introduction to Ergodic Theory}\/} (Springer-Verlag, New
  York-Berlin, 1982).

\bibitem{WangZhao2018}
Q.~Wang and Y.~Zhao.
\newblock Variational principle and zero temperature limits of asymptotically
  (sub)-additive projection pressure.
\newblock {\em Front. Math. China\/} {\bf 13} (2018), 1099--1120.

\bibitem{Yan_2013}
K.~Yan.
\newblock Sub-additive and asymptotically sub-additive topological pressure for
  {$\Bbb{Z}^d$}-actions.
\newblock {\em J. Dyn. Diff. Equat.\/} {\bf 25} (2013), 653--678.

\bibitem{yayama_2011b}
Y.~Yayama.
\newblock Applications of a relative variational principle to dimensions of
  nonconformal expanding maps.
\newblock {\em Stoch. Dyn.\/} {\bf 11} (2011), 643--679.

\bibitem{Yayama_2011}
Y.~Yayama.
\newblock Existence of a measurable saturated compensation function between
  subshifts and its applications.
\newblock {\em Ergodic Theory Dynam. Systems\/} {\bf 31} (2011), 1563--1589.

\bibitem{Yayama_2016}
Y.~Yayama.
\newblock On factors of {G}ibbs measures for almost additive potentials.
\newblock {\em Ergodic Theory Dynam. Systems\/} {\bf 36} (2016), 276--309.

\bibitem{Yuri_2003}
M.~Yuri.
\newblock Weak {G}ibbs measures for intermittent systems and weakly {G}ibbsian
  states in statistical mechanics.
\newblock {\em Commun. Math. Phys.\/} {\bf 241} (2003), 453--466.

\bibitem{zhao_conditional_2014}
Y.~Zhao.
\newblock Conditional ergodic averages for asymptotically additive potentials.
\newblock {\em arXiv:1405.1648\/}.

\bibitem{1078_0947_2019_4_2059}
Y.~Zhao, W.-C. Cheng, and C.-C. Ho.
\newblock Q-entropy for general topological dynamical systems.
\newblock {\em Discrete Contin. Dyn. Syst. Ser. A\/} {\bf 39} (2019),
  2059--2075.

\bibitem{zhao_asymptotically_2011}
Y.~Zhao, L.~Zhang, and Y.~Cao.
\newblock The asymptotically additive topological pressure on the irregular set
  for asymptotically additive potentials.
\newblock {\em Nonlinear Anal.\/} {\bf 74} (2011), 5015--5022.

\end{thebibliography}
\end{document}